\documentclass[preprint,12pt]{elsarticle}

\usepackage{graphicx}
\usepackage{amssymb}
\usepackage{amsmath}
\usepackage{mathtools}
\usepackage{seqsplit}
\usepackage[utf8]{inputenc}
\usepackage{float}

\usepackage{lineno}
\usepackage{amsthm}
\usepackage{dsfont}
\usepackage{hyperref}
\usepackage{tikz}
\usetikzlibrary{positioning}
  
\tikzset{sgplattice/.style={inner sep=1pt,norm/.style={red!50!blue},char/.style={blue!50!black},
  lin/.style={black!50}},cnj/.style={black!50,yshift=-2.5pt,left=-1pt of #1,scale=0.5,fill=white}}

\newcommand{\Ind}{\operatorname{Ind}}
\newcommand{\RO}{\operatorname{RO}}

\newcommand{\GL}{\operatorname{GL}}

\newcommand{\Aut}{\operatorname{Aut}}
\newcommand{\tr}{\operatorname{tr}}

\newcommand{\Res}{\operatorname{Res}}

\newcommand{\Mod}[1]{\ (\mathrm{mod}\ #1)}

\newcommand{\rank}{\operatorname{rank}}
\newcommand{\PO}{\operatorname{PO}}
\newcommand{\prim}{\operatorname{prim}}

\newcommand{\Ker}{\operatorname{Ker}}
\newcommand\blfootnote[1]{%
  \begingroup
  \renewcommand\thefootnote{}\footnote{#1}%
  \addtocounter{footnote}{-1}%
  \endgroup
}
\newtheorem{theorem}{Theorem}[section]

\newtheorem{lemma}[theorem]{Lemma}
\newtheorem{corollary}[theorem]{Corollary}
\newtheorem{question}[theorem]{Question}

\newtheorem{remark}[theorem]{Remark}
\newtheorem{conjecture}[theorem]{Conjecture}
\newtheorem{problem}[theorem]{Problem}
\newtheorem*{remark*}{Remark}

\newtheorem*{theorem*}{Theorem}
\newtheorem*{conjecture*}{Conjecture}
\newtheorem*{question*}{Question}

\numberwithin{table}{section}

\journal{Journal Name}

\makeatletter
\def\ps@pprintTitle{%
 \let\@oddhead\@empty
 \let\@evenhead\@empty
 \def\@oddfoot{}%
 \let\@evenfoot\@oddfoot}
\makeatother

\DeclareMathOperator{\Ima}{Im}
\DeclareMathOperator{\Rea}{Re}

\begin{document}
\begin{frontmatter}


\title{\textbf{A new family of finite Oliver groups satisfying the Laitinen Conjecture}}



\author{Piotr Mizerka}

\begin{abstract}
This paper is concerned with the Laitinen Conjecture. The conjecture predicts an answer to the Smith question \cite{Smith1960} which reads as follows. Is it true that for a finite group acting smoothly on a sphere with exactly two fixed points, the tangent spaces at the fixed points have always isomorphic group module structures defined by differentiation of the action? Using the technique of induction of group representations, we indicate a new infinite family of finite Oliver groups for which the Laitinen Conjecture holds.
\end{abstract}




\end{frontmatter}

\linenumbers
\setcounter{section}{-1}
\nolinenumbers
\section{Introduction}
\blfootnote{
$2010$ \emph{Mathematics Subject Classification.} Primary $57\rm S25$; Secondary $55\rm M35$.

\emph{\textcolor{white}{ggf}Keywords and phrases.} induced character, fixed point, Smith problem, smooth action.}
P. A. Smith raised in $1960$ the following question for finite groups \cite[p. 406, the footnote]{Smith1960}. 
\begin{question}[Smith question]
Is it true that for a finite group acting smoothly on a sphere with exactly two fixed points, the tangent spaces at the fixed points have always isomorphic group module structures defined by differentiation of the action?
\end{question}
Depending on the acting group, there are affirmative, as well as negative answers to this question. Much of the work on this problem has been done by Atiyah and Bott \cite{AtiyahBott1968}, Petrie and his students and collaborators \cite{Petrie1978},\cite{Petrie1979},\cite{Petrie1982} and \cite{Petrie1983}, Cappell and Shaneson \cite{CappellShaneson1980}\cite{Cappell1982}, \cite{Bredon1969}, Illman \cite{Illman1982}, Milnor \cite{Milnor1966}, Laitinen, Morimoto, Pawałowski, Solomon and Sumi, \cite{LaitinenPawalowski1999}, \cite{Morimoto1998},\cite{Morimoto2010}, \cite{PawalowskiSolomon2002}, \cite{MorimotoPawalowski2003}, \cite{Sumi2016}. For a comprehensive survey on the Smith problem, we refer the reader to the work of Pawałowski \cite{Pawalowski2018}.

Assume $G$ is a finite group. Let us call two $\mathbb{R}G$-modules $U$ and $V$ \emph{Smith equivalent} if $U\cong T_x(\Sigma)$ and $V\cong T_y(\Sigma)$ for a smooth action of $G$ on a homotopy sphere $\Sigma$ with exactly two fixed points $x$ and $y$. We say that the \emph{Laitinen Condition} is satisfied for $G$ acting smoothly on a homotopy sphere $\Sigma$ with $\Sigma^G=\{x,y\}$ if $\Sigma^g$ is connected for any $g\in G$ of order $2^k$, where $k\geq 3$. The \emph{real conjugacy class} of an element $g\in G$ is the union $(g)^{\pm}=(g)\cup(g^{-1})$. The \emph{primary number} of $G$ which we denote by $\prim(G)$ is the number of real conjugacy classes of $G$ containing elements whose order is divisible by at least two distinct primes. We call $G$ an \emph{Oliver group} if there does not exist a sequence of subgroups $P\trianglelefteq H\trianglelefteq G$ such that $P$ and $G/H$ are of prime power orders and $H/P$ is cyclic. 

The Laitinen Conjecture proposes negative answers to the Smith question concerning actions on homotopy spheres. The conjecture reads as follows.
\begin{conjecture}\cite[Appendix]{LaitinenPawalowski1999}\label{conjecture:laitinen}
If $G$ is an Oliver group with $\prim(G)\geq 2$, then there exist non-isomorphic $\mathbb{R}G$-modules $U$ and $V$ which are Smith equivalent and the action of $G$ on the homotopy sphere in question satisfies the Laitinen Condition.
\end{conjecture}
The converse conclusion is always true \cite{LaitinenPawalowski1999} and Conjecture \ref{conjecture:laitinen} is known to be true in the following cases, \cite{Pawalowski2018}.
\begin{enumerate}[(1)]
    \item $G$ is of odd order (and thus, by the Feit-Thompson Theorem, $G$ is solvable).
    \item $G$ has a cyclic quotient of odd composite order (for example, $G$ is a nilpotent group with three or more noncyclic Sylow subgroups).
    \item $G$ is a nonsolvable group not isomorphic to $\Aut(A_6)$ (in the case where $G = {\rm Aut}(A_6)$, the Laitinen Conjecture is false by \cite{Morimoto2008}).
    \item $G$ satisfies the Sumi $G^{\rm nil}$-condition (the condition is defined below).
\end{enumerate}
For a prime $p$, let us use the notation $\mathcal{O}^p(G)$ for the smallest normal subgroup of $G$ with $G/\mathcal{O}^p(G)$ a $p$-group. A subgroup $H$ of a group $G$ is called \emph{large} if $\mathcal{O}^p(G)\leq H$ for some prime $p$. We denote by $\mathcal{L}(G)$ the family of all large subgroups of $G$. Let us call $G$ a \emph{gap group} if there exists an $\mathbb{R}G$-module $V$ such that for any $P<H\leq G$ with $P$ of prime power order, we have $\dim V^P>2\dim V^H$ and for any $L\in\mathcal{L}(G)$, $\dim V^L=0$ holds. Denote by $G^{\rm nil}$ the smallest normal subgroup of $G$ such that $G/G^{\rm nil}$ is nilpotent. We say that $G$ satisfies the \emph{Sumi $G^{\rm nil}$-condition} if there exist two elements $a,b\in G$ of composite order which are not real conjugate in $G$, the equality $aG^{\rm nil}=bG^{\rm nil}$ holds and at least one of the following statements holds.
\begin{itemize}
    \item $|a|$ and $|b|$ are even and the involutions of the cyclic subgroups $\langle a\rangle$ and $\langle b\rangle$ are conjugate in $G$.
    \item $a$ and $b$ belong to the same gap subgroup of $G$.
\end{itemize}

Therefore, in checking the Laitinen Conjecture, we shall focus on finite solvable Oliver groups $G$ of even order, such that each cyclic quotient of $G$ is either of even or of prime power order and $G$ does not satisfy the Sumi $G^{\rm nil}$-condition.  We refer to such a group $G$ as a \emph{special Oliver group}. In general, however, Conjecture \ref{conjecture:laitinen} is not true. It fails for example for $G=\Aut(A_6)$ \cite{Morimoto2008} or $G=S_3\times A_4$ (see \cite{PawalowskiSumi2009} for more counterexamples). 

We say that two $\mathbb{R}G$-modules $U$ and $V$ are \emph{$\mathcal{P}$-matched} if for any subgroup $P\leq G$ of prime power order, the restrictions $\Res^G_P(U)$ and $\Res^G_P(V)$ are isomorphic as $P$-modules. 

In $2018$, Pawałowski \cite{Pawalowski2018} proposed the following problem.
\begin{problem}\label{question:pawalowski}
For which special Oliver groups $G$ with $\prim(G)\geq 2$, there exist $\mathcal{P}$-matched and Smith equivalent $\mathbb{R}G$-modules $U$ and $V$ which are not isomorphic?
\end{problem}
Some examples of special Oliver groups $G$ with $\prim(G)\geq 2$ such that no $\mathbb{R}G$-modules in question exist were already given in \cite{PawalowskiSumi2009}. We present here, for the first time, a certain infinite family of special Oliver groups with primary numbers at least $2$, possessing pairs of $\mathcal{P}$-matched Smith equivalent $\mathbb{R}G$-modules which are not isomorphic. 

Suppose $p$ and $q$ are odd prime numbers such that $q|(p-1)$. Let $D_{2pq}$ be the dihedral group of order $2pq$ and $C_q$ be the cyclic group of order $q$. These groups have the following presentations.
$$
\begin{tabular}{ccc}
$
D_{2pq}=\langle a,b|a^{pq}=b^2=1,bab=a^{-1}\rangle
$&
and&
$
C_q=\langle c|c^q=1\rangle.
$
\end{tabular}
$$
Let $v$ be a primitive root modulo $p$ which is not divisible by $q$ (in case $q|v$, just take $p+v$ instead of $v$ which is also a primitive root modulo $p$). Put $i=v^{(p-1)(q-1)/q}\Mod{pq}$ and note that $i\equiv 1\Mod{q}$ and the order of $i$ modulo $p$ is $q$. Therefore $i\not\equiv 1\Mod{pq}$ and $i^q\equiv 1\Mod{pq}$ by the Chinese Reminder Theorem. Consider the automorphism $\tau$ of $D_{2pq}$ given by $\tau(a)=a^i$ and $\tau(b)=b$. The order of $\tau$ is $q$. Thus, we have a homomorphism $\varphi\colon C_q\rightarrow\Aut(D_{2pq})$, $c\mapsto\tau$. Define $G_{p,q}$ as the following semidirect product.
$$
G_{p,q}=D_{2pq}\rtimes_{\varphi}C_q
$$
The main theorem of the article can be stated as follows.
\begin{theorem}\label{theorem:main}
For any odd primes $p$ and $q$ with $q|(p-1)$, $G_{p,q}$ is a special Oliver group with $\prim(G_{p,q})\geq 2$. Moreover, there exist non-isomorphic $\mathcal{P}$-matched Smith equivalent $\mathbb{R}G_{p,q}$-modules $U$ and $V$.
\end{theorem}
\begin{remark}\emph{Note that the theorem above confirms the Laitinen Conjecture for $G_{p,q}$'s since the Laitinen Condition is naturally satisfied due to the lack of elements of order divisible by $8$ in $G_{p,q}$'s.}
\end{remark}
\begin{remark}
\emph{In the case where $q=2$, $N_p=\{(a^{qs},1)|s=0,...,p-1\}$ is a normal subgroup of $G_{p,q}$ isomorphic to the cyclic group of order $p$, such that the quotient $G_{p,q}/N_p$ is a $2$-group. Thus, for $q = 2$, $G_{p,q}$ is not an Oliver group. Moreover, any nontrivial element of $G_{p,q}$ is of order 2, 4, or $p$, where $p$ is an odd prime. Therefore, by elementary character theory arguments and the result of Atiyah and Bott [1, Thm. 7.15], any two Smith equivalent $\mathbb{R}G_{p,q}$-modules are isomorphic.}
\end{remark}
Fix odd primes $p$ and $q$ such that $q|(p-1)$. For the better presentation of the material, let us introduce additionally the following symbols and concepts ($G$ denotes a finite group).
\begin{itemize}
    \item $\RO(G)$ - the real representation group of $G$. Consists of formal differences $U-V$ of $\mathbb{R}G$-modules. We identify $U-V$ with $U'-V'$ if there exists an $\mathbb{R}G$-module $W$ such that $U\oplus V'\oplus W\cong U'\oplus V\oplus W$. The addition is induced by direct sum operation.
    \item $\PO(G)$ - the subgroup of $\RO(G)$ containing elements $U-V$ such that $U$ and $V$ are $\mathcal{P}$-matched.
    \item An $\mathbb{R}G$-module $V$ is said to satisfy the \emph{weak gap condition} if for any $P<H\leq G$ such that $P$ is of prime power order, we have $\dim V^P\geq 2\dim V^H$.
    \item $\PO^{\mathcal{L}}_{\rm w}(G)$ - the subgroup of $\PO(G)$ containing elements which can be written as $U-V$ for some $\mathbb{R}G$-modules $U$ and $V$ satisfying the weak gap condition and such that $\dim W^L=0$ for any $L\in\mathcal{L}(G)$ and $W=U,V$.
    \item $N_{pq^2}$ - the unique subgroup of $G_{p,q}$ of index $2$.
    \item $\Ind^G_H\colon\PO(H)\rightarrow\PO(G)$ - the induction homomorphism defined for any subgroup $H\leq G$ by the formula $U-V\mapsto\Ind^G_H(U)-\Ind^G_H(V)$, where $\Ind^G_H(W)$ denotes the induced $\mathbb{R}G$-module from the $\mathbb{R}H$-module $W$. This is a well-defined map since, if $U$ and $V$ are $\mathcal{P}$-matched $\mathbb{R}H$-modules, then so are $\Ind^G_H(U)-\Ind^G_H(V)$ as $\mathbb{R}G$-modules (we comment on this fact in the subsequent part).
\end{itemize}The paper is organized as follows. First, we show that $\PO^{\mathcal{L}}_{\rm w}(N_{pq^2})\neq 0$. In the next section, we prove that $G_{p,q}$ is a special Oliver group with $\prim(G_{p,q})\geq 2$. The third section provides, for any finite groups $H\leq G$, the necessary and sufficient condition for $\Ind^G_H\colon\PO(H)\rightarrow\PO(G)$ to be a monomorphism. Finally, we prove Theorem \ref{theorem:main} using the properties of the induction from $N_{pq^2}$ to $G_{p,q}$. 

From now on, all groups considered in this paper are assumed to be finite.
\section{$\PO^{\mathcal{L}}_{\rm w}(N_{pq^2})$ is nonzero}
Note that $N_{pq^2}=\{(a^l,c^m)|k=0,...,pq-1, m=0,...,q-1\}$. We have
$$
(1,c)(a,1)(1,c)^{-1}=(1,c)(a,1)(1,c^{-m})=(1,c)(a,c^{-m})=(a^i,1).
$$
Thus, under the identifications $a\leftrightarrow(a,1)$ and $c\leftrightarrow(1,c)$, $N_{pq^2}$ can be presented as $$N_{pq^2}=\langle a,c|a^{pq}=c^q=1,cac^{-1}=a^i\rangle.$$ 
Let
$$
N_{pq^2}'=\langle \alpha,\beta,\gamma|\alpha^{q}=\beta^p=\gamma^q=1,\gamma\beta\gamma^{-1}=\beta^i,\alpha\beta=\beta\alpha,\alpha\gamma=\gamma\alpha\rangle.
$$
Then $N'_{pq^2}$ is isomorphic to the direct product of $C_q=\langle\alpha\rangle$ with the Frobenius group $F_{p,q}$ generated by $\beta$ and $\gamma$. 
\begin{lemma}\label{lemma:isom}
Let $f\colon N'_{pq^2}\rightarrow N_{pq^2}$ be given by $f(\alpha)=a^p$, $f(\beta)=a^q$ and $f(\gamma)=c$. Then $f$ is a group isomorphism.
\end{lemma}
\begin{proof}
Note that $f$ is a well-defined group homomorphism. Indeed, $f(\alpha^q)=a^{pq}=1$, $f(\beta^p)=a^{pq}=1$, $f(\gamma^q)=c^q=1$, $f(\gamma\beta\gamma^{-1})=ca^qc^{-1}=(cac^{-1})^q=a^{iq}=f(\beta^i)$, $f(\alpha\beta)=a^{p+q}=a^{q+p}=f(\beta\alpha)$, $f(\gamma\alpha\gamma^{-1})=ca^pc^{-1}=a^{pi}=a^p=f(\alpha)$. The equality $a^{pi}=a^p$ follows from the fact that $pq|p(i-1)$ since $i\equiv 1\Mod{q}$.

Take any $a^lc^m\in N_{pq^2}$. Since $p$ and $q$ are different primes, we can find $x,y\in\mathbb{Z}$ such that $1=xp+yq$ and
$$
f(\alpha^{lx}\beta^{ly}\gamma^m)=a^{plx}a^{qly}c^m=a^{l(xp+yq)}c^m=a^lc^m.
$$
Hence $f$ is surjective. Let us prove that it is injective as well. Suppose $f(\alpha^x\beta^yc^m)=1$. Then $a^{px+qy}c^m=1$ which is the case only if $pq|(px+qy)$ and $m=0$. Since $p|px$ and $q|qy$, we must have then $p|qy$ and $q|px$ but this means $p|y$ and $q|x$. As a consequence, $\alpha^x\beta^y\gamma^m=1$ and $f$ has the trivial kernel.
\end{proof}
Put $u=i\Mod{p}$, $r=(p-1)/q$ and $\mathbb{Z}_p^*/\langle u\rangle=v_1\langle u\rangle\cup...\cup v_r\langle u\rangle$. Following \cite{Liebeck2001}[25.10 Theorem] the conjugacy classes of $F_{p,q}$ are as follows.
\begin{table}[H]
$$
\begin{tabular}{c|ccc}
     class&$(1)$&$(\beta^{v_j})$&$(\gamma^n)$  \\
     \hline
     representative order&$1$&$p$&$q$\\
     \hline
     size&$1$&$q$&$p$\\
     \hline
     \# of classes of a given type&$1$&$r$&$q-1$
\end{tabular}
$$
\caption{\label{table:conjugacyClassesFrobenius}Conjugacy classes of $F_{p,q}$.}
\end{table}
where $(\beta^{v_j})=\{\beta^{v_js}|s\in\langle u\rangle\}$ and $(b^n)=\{\beta^m\gamma^n|0\leq m\leq p-1\}$ for all $1\leq j\leq r$ and $1\leq n\leq q-1$. Let $\sigma_{t,x}=\sum_{s\in\langle u\rangle}e^{2\pi iv_txs/p}$ for $x=0,...,p-1$. We have $r$ nonlinear irreducible characters of $F_{p,q}$ given by $\chi_t(\beta^x)=\sigma_{t,x}$ and $\chi_t(\gamma^n)=0$ for $x=0,...,p-1$ and $n=1,...,q-1$. They are presented in the table below.
\begin{table}[H]
$$
\begin{tabular}{c|ccc}
     &$(1)$&$(\beta^{v_j})$&$(\gamma^n)$  \\
     \hline
     $\chi_1$&$q$&$\sigma_{1,v_j}$&$0$\\
     \hline
     $\vdots$&$\vdots$&$\vdots$&$\vdots$\\
     \hline
     $\chi_r$&$q$&$\sigma_{r,v_j}$&$0$
\end{tabular}
$$
\caption{\label{table:charactersFrobenius}Nonlinear irreducible characters of $F_{p,q}$.}
\end{table}
The irreducible characters of $C_q$ are $\rho_s\colon C_q\rightarrow\mathbb{C}$, $\alpha\mapsto\zeta_q^s$ for $s=0,...,q-1$, where $\zeta_q=e^{2\pi i/q}$. Since the irreducible characters of direct products are products of irreducible characters of the factor groups \cite{Liebeck2001}[19.18 Theorem], the following table contains the nonlinear irreducible characters of $N_{pq^2}\cong C_q\times F_{p,q}$.
\begin{table}[H]
$$
\begin{tabular}{c|ccccccc}
     $g$&$(1,1)$&$(1,\beta^{v_j})$&$(1,\gamma^{n})$&$(\alpha^l,\beta^{v_j})$&$(\alpha^l,\gamma^n)$&$(\alpha^l,1)$&  \\
     \hline
     $|g|$&$1$&$p$&$q$&$pq$&$q$&$q$\\
     $|(g)|$&$1$&$q$&$p$&$q$&$p$&$1$\\
     \# $(g)$&$1$&$r$&$q-1$&$(q-1)r$&$(q-1)^2$&$q-1$\\
     \hline
     $\psi_{s,t}=\rho_s\times\chi_t$&$q$&$\sigma_{t,v_j}$&$0$&$\zeta_q^{ls}\sigma_{t,v_j}$&$0$&$q\zeta_q^{ls}$
\end{tabular}
$$
\caption{\label{table:charactersNpq2}Nonlinear irreducible characters of $N_{pq^2}$.}
\end{table}
Let $N_p=\{(1,\beta^s)|s=0,...,p-1\}$. Obviously, $N_p$ is a normal subgroup of $N_{pq^2}$ isomorphic to $C_p$.
\begin{lemma}\label{lemma:largeNpq2}
$\mathcal{O}^q(N_{pq^2})=N_p$ and $\mathcal{O}^p(N_{pq^2})=N_{pq^2}$. As a result, all $L\in\mathcal{L}(N_{pq^2})$ contain $N_p$ as a subgroup.
\end{lemma}
\begin{proof}
It is obvious that $\mathcal{O}^q(N_{pq^2})=N_p$. We show that there is no normal subgroup of $N_{pq^2}$ of order $q^2$ which would conclude the proof.

Suppose for the converse that $N$ is a normal subgroup of $N_{pq^2}$ of order $q^2$. There exists $g\in N$ of order $q$. Since $N\trianglelefteq N_{pq^2}$, we have $(g)\subseteq N$. We know by Table \ref{table:charactersNpq2} that $g$ belongs to one of the following conjugacy classes: $((1,\gamma^n))$, $((\alpha^l,\gamma^n))$ or $((\alpha^l,1))$. Suppose $g\in ((1,\gamma^{n_0}))$ for some $n_0\in\{1,...,q-1\}$. Since $\{(1,\gamma^n)|n=0,...,q-1\}=\langle(1,\gamma^{n_0})\rangle\leq N$, it follows that each class $((1,\gamma^n))$ is contained in $N$. This yields at least $p(q-1)>q^2$ elements in $N$. A contradiction. Let $g\in((\alpha^{l_0},\gamma^{n_0}))$ for some $l_0,n_0\in\{1,...,q-1\}$. Then, similarly as before, considering $\langle(\alpha^{l_0},\gamma^{n_0})\rangle\leq N$ yields at least $p(q-1)>q^2$ elements in $N$ and we can exclude this case as well. Thus, all elements of order $q$ of $N$ belong to one of the classes $((\alpha^l,1))$. From Table \ref{table:charactersNpq2} follows that there are $q-1$ elements in these classes. Moreover, every element of $N$ different from the identity is of order $q$. This yields $|N|=q$ which is also a contradiction. 
\end{proof}
Since characters of any group $G$ determine $FG$-modules up to isomorphism for $F=\mathbb{R},\mathbb{C}$, we shall use the same symbols for the characters and the $FG$-modules determined by them. Moreover, if $\chi$ is the character of $G$ determined by some $FG$-module, then by $\dim\chi^H$ we mean the fixed point dimension over $F$ for a subgroup $H$ acting on this $FG$-module. Note that in case $\chi$ is a character of some $\mathbb{R}G$-module then all such fixed point dimensions over $\mathbb{R}$ are equal as considered over $\mathbb{C}$ - we can treat $\chi$ as a character of a $\mathbb{C}G$-module as well.
\begin{lemma}\label{lemma:zeroDimension}
Let $s\neq 0$ and $H$ be a subgroup of $N_{pq^2}$ of order $p$ or $q^2$. Then $\dim\psi_{s,t}^H=0$ for any $t=0,...,r$.
\end{lemma}
\begin{proof}
Suppose $|H|=p$. Then $H=N_p$ and it follows from Table \ref{table:charactersNpq2} that
\begin{eqnarray*}
\dim\psi_{s,t}^H&=&\frac{1}{|H|}\sum_{h\in H}\psi_{s,t}(h)=\frac{1}{p}\Big(q+\sum_{x=1}^{p-1}\sigma_{t,x}\Big)=\frac{1}{p}\Big(q+\sum_{x=1}^{p-1}\sum_{s\in\langle u\rangle}e^{2\pi iv_txs/p}\Big)\\
&=&\frac{1}{p}\Big(q+\sum_{s\in\langle u\rangle}\sum_{x=1}^{p-1}e^{2\pi iv_txs/p}\Big)=\frac{1}{p}\Big(q+\sum_{s\in\langle u\rangle}(-1)\Big)=0.
\end{eqnarray*}
If $|H|=q^2$, then, since the only nonzero values of $\psi_{s,t}$ on elements of order $q$ are taken for the classes $(\alpha^l,1)$, it follows that
$$
\dim\psi_{s,t}^H\leq\frac{1}{q^2}\Big(q+\sum_{l=1}^{q-1}|q\zeta_q^{ls}|\Big)<\frac{1}{q^2}(q+q(q-1))=1
$$
and $\dim\psi_{s,t}^H=0$.
\end{proof}
\begin{corollary}\label{corollary:weakGapNpq2}
If $s\neq 0$, then $2\Rea\psi_{s,t}$ is an $\mathbb{R}N_{pq^2}$-module satisfying the weak gap condition and such that $\dim(2\Rea\psi_{s,t})^L=0$ for any $L\in\mathcal{L}(N_{pq^2})$.
\end{corollary}
\begin{proof}
From the properties of real and complex irreducible representations, we know that $2\Rea\psi_{s,t}$ is the character of a real irreducible $N_{pq^2}$-module since $\psi_{s,t}$ is not real-valued. Take any $L\in\mathcal{L}(N_{pq^2})$. We know by Lemma \ref{lemma:largeNpq2} that $N_p\leq L$. Thus, by Lemma \ref{lemma:zeroDimension}, we get
$$
\dim(2\Rea\psi_{s,t})^L=\dim(\psi_{s,t}+\overline{\psi_{s,t}})^L=2\dim\psi_{s,t}^L\geq 2\dim\psi_{s,t}^{N_p}=0.
$$
It remains to show that $2\Rea\psi_{s,t}$ satisfies the weak gap condition. By means of Lemma \ref{lemma:zeroDimension}, this boils down to proving that
$$
\dim(2\Rea\psi_{s,t})\geq 2\dim(2\Rea\psi_{s,t})^H
$$
for any subgroup $H\leq N_{pq^2}$ of order $q$. Using, once again, the information from Table \ref{table:charactersNpq2}, we get
$$
\dim(2\Rea\psi_{s,t})=2q>4\geq 2\cdot 2\cdot\frac{1}{q}(1+q-1)\geq 2\dim(2\Rea\psi_{s,t})^H.
$$
\end{proof}
\begin{lemma}\label{lemma:noznzeroPO}
Let $s\neq 0$. Then, for any $t=1,...,r$, the $\mathbb{R}G$-modules $U=2\Rea\psi_{s,t}$ and $V=2\Rea\psi_{q-s,t}$ are not isomorphic and $\mathcal{P}$-matched. As a result, $\PO_{\rm w}^{\mathcal{L}}(N_{pq^2})\neq 0$.
\end{lemma}
\begin{proof}
It follows from Table \ref{table:charactersNpq2} that $U$ and $V$ are $\mathcal{P}$-matched. Note that $U=\rho_s\times\chi_t$ and $V=\overline{\rho_s}\times\chi_t$. By the similar computation as in the proof of Lemma \ref{lemma:zeroDimension}, we establish the Frobenius-Schur indicator of character $\chi_t$.
\begin{eqnarray*}
\iota(\chi_t)&=&\frac{1}{|F_{p,q}|}\sum_{g\in F_{p,q}}\chi_{t}(g^2)=\frac{1}{pq}\Big(q+\sum_{|g|=p}\chi_{t}(g^2)\Big)=\frac{1}{pq}\Big(q+\sum_{|g|=p}\chi_t(g)\Big)\\
&=&\frac{1}{pq}\Big(q+\sum_{x=1}^{p-1}\sigma_{t,x}\Big)=0.
\end{eqnarray*}
Thus, $\chi_t$ is not real-valued and we can take $x=0,...,p-1$ such that $\Ima(\chi_t(\beta^x))\neq 0$. Now, take $l=1,...,q-1$ and put $g=(\alpha^l,\beta^{x})$. Clearly, $g$ is an element of order $pq$. Then, the character of $U$ evaluated on $g$ is equal to the number
\begin{eqnarray*}
\chi_U(g)&=&2\Rea\psi_{s,t}(g)=2\Rea(\rho_s(\alpha^l)\chi_t(\beta^x))\\
&=&2(\Rea(\rho_s(\alpha^l))\Rea(\chi_t(\beta^x))-\Ima(\rho_s(\alpha^l))\Ima(\chi_t(\beta^x))).
\end{eqnarray*}
Analogously, the character of $V$ evaluated on $g$ is equal to
\begin{eqnarray*}
\chi_V(g)&=&2\Rea\psi_{q-s,t}(g)=2\Rea(\overline{\rho_s(\alpha^l)}\chi_t(\beta^x))\\
&=&2(\Rea(\rho_s(\alpha^l))\Rea(\chi_t(\beta^x))+\Ima(\rho_s(\alpha^l))\Ima(\chi_t(\beta^x)))\\
&\neq& \chi_U(g).
\end{eqnarray*}
\end{proof}
\section{$G_{p,q}$ is a special Oliver group with $\prim(G_{p,q})\geq 2$}
We divide the material contained in this section into three parts. In the first, we determine conjugacy classes of $G_{p,q}$. Using this, we show that $\prim(G_{p,q})\geq 2$. In the next part, we establish all normal subgroups of $G_{p,q}$ and infer the necessary information concerning the quotients of $G_{p,q}$. Finally, we use the performed computations to prove that $G_{p,q}$ is an example of a special Oliver group.
\subsection{Conjugacy classes of $G_{p,q}$}
Any element of $G_{p,q}$ is either of the form $x_1=(ba^l,c^m)$ or $x_2=(a^l,c^m)$ for some $l=0,...,pq-1$ and $m=0,...,q-1$. In the first case, its inverse $x_1^{-1}=(ba^{li^{-m}},c^{-m})$, while in the second $x_2^{-1}=(a^{-li^{-m}},c^{-m})$.

Let $g\in G_{p,q}$. We have the following possibilities.
\begin{enumerate}[(1)]
    \item $g=(ba^{l_0},c^{m_0})$. Then 
    $$
    \begin{tabular}{ccc}
    $x_1gx_1^{-1}=(ba^{l(1+i^{m_0})-l_0i^m},c^{m_0})$&and&
    $x_2gx_2^{-1}=(ba^{-l(1+i^{m_0})+l_0i^m},c^{m_0}).$
    \end{tabular}
    $$
    Note that the expression $l(1+i^{m_0})-l_0i^m$ can take any remainder modulo $pq$. Since $i\equiv 1\Mod{q}$, it follows that $l(1+i^{m_0})-l_0i^m\equiv 2l-l_0\Mod{q}$ and substituting subsequent values for $l=0,...,pq-1$, we can obtain any pair of remainders of $l(1+i^{m_0})-l_0$ modulo $p$ and $q$. We conclude then from the Chinese Remainder Theorem, that for any $l'=0,...,pq-1$, there exist $l=0,...pq-1$ such that $l(1+i^{m_0})-l_0i^0=l(1+i^{m_0})-l_0\equiv l'\Mod{pq}$. Therefore
    $$
    (g)=\{(ba^l,c^{m_0})|l=0,...,pq-1\}.
    $$
    Note that $(b,c^{m_0})^n=(b^n,c^{nm_0})$. Hence $|g|=2q$ if $m_0\neq 0$ and $|g|=2$ if $m_0=0$.
    \item $g=(a^{l_0},c^{m_0})$, where $m_0\neq 0$. Then
    $$
    \begin{tabular}{ccc}
    $x_1gx_1^{-1}=(a^{l(i^{m_0}-1)-l_0i^m},c^{m_0})$&and&
    $x_2gx_2^{-1}=(a^{-l(i^{m_0}-1)+l_0i^m},c^{m_0}).$
    \end{tabular}
    $$
    We have $l(i^{m_0}-1)-l_0i^m\equiv -l_0\Mod{q}$ and substituting subsequent values for $l$, we can achieve all remainders modulo $p$ of $l(i^{m_0}-1)-l_0i^m$. If $r_0$ is the remainder modulo $q$ of $l_0$, it follows then that
    $$
    (g)=\{(a^{r_0+lq},c^{m_0}),(a^{-r_0+lq},c^{m_0})|l=0,...,p-1\}.
    $$
    For any $n\geq 0$
    $$
    g^n=(a^{l_0(1+i^{m_0}+...+i^{(n-1)m_0})},c^{m_0})=(a^{l_0\cdot\frac{1-i^{nm_0}}{1-i^{m_0}}},c^{nm_0}).
    $$
    Thus $q||g|$. On the other hand $p|\frac{1-i^{qm_0}}{1-i^{m_0}}$ since $1-i^{qm_0}$ is divisible by $pq$ and $p\nmid 1-i^{m_0}$. Moreover, $1+i^{m_0}+...+i^{(q-1)m_0}\equiv q\equiv 0\Mod{q}$, so $pq|\frac{1-i^{qm_0}}{1-i^{m_0}}$ and $|g|=q$.
    \item $g=(a^{l_0},1)$. The computations of conjugacy class elements reduce then to
    $$
    \begin{tabular}{ccc}
    $x_1gx_1^{-1}=(a^{-l_0i^m},1)$&and&
    $x_2gx_2^{-1}=(a^{l_0i^m},1).$
    \end{tabular}
    $$
    If $p\nmid l_0$, then all the numbers from the set $S_{l_0}=\{\pm l_0i^m|m=0,...,q-1\}$ give different remainders modulo $p$ - this follows from the definition of $i$. Thus, we have $(p-1)/2$ such conjugacy classes, each with $2q$ elements and $$(g)=\{(a^{l_0i^m},1),(a^{-l_0i^m},1)|m=0,...,q-1\}$$
    Moreover, for any $n\geq 0$, $g^n=(a^{nl_0},1)$, so $|g|=pq$ if $q\nmid l_0$ and $|g|=p$ if $q|l_0$.
    
    If $p|l_0$ and $q\nmid l_0$, then the set $S_{l_0}$ reduces to two elements, $(a^{l_0},1)$ and $(a^{-l_0},1)$. We have $(q-1)/2$ such classes and
    $$
    \begin{tabular}{ccc}
    $(g)=\{(a^{l_0},1),(a^{-l_0},1)\}$&and&
    $|g|=q.$
    \end{tabular}
    $$
    Finally, the last class left is the class of the identity element, $(g)=\{(1,1)\}$.
\end{enumerate}
The following table summarizes the information about the conjugacy classes of $G_{p,q}$ and orders of its elements.
\begin{table}[H]
$$
\begin{tabular}{c|cccccccc}
     $g$&$(1,1)$&$B$&$E_s$&$C_{m}$&$D_{r,m}$&$F_s$&$B_m$&$A_l$ \\
     \hline
     $|g|$&$1$&$2$&$p$&$q$&$q$&$q$&$2q$&$pq$\\
     $|(g)|$&$1$&$pq$&$2q$&$p$&$2p$&$2$&$pq$&$2q$\\
     \# $(g)$&$1$&$1$&$\frac{1}{2}r$&$q-1$&$\frac{1}{2}(q-1)^2$&$\frac{1}{2}(q-1)$&$q-1$&$\frac{1}{2}(q-1)r$
\end{tabular}
$$
\caption{\label{table:classesGpq}Conjugacy classes of $G_{p,q}$.}
\end{table}
where
$$
    \begin{tabular}{cc}
    $B=(b,1),$&$E_s=(a^{qs},1), s=1,...,p-1,$\\
    $C_m=(1,c^m), m=1,...,q-1,$&$D_{r,m}=(a^r,c^m), m=1,...,q-1, q\nmid r,$\\
    $F_s=(a^{ps},1), s=1,...,q-1,$&$B_m=(b,c^m), m=1,...,q-1,$\\
    $A_l=(a^l,1), p,q\nmid l.$
    \end{tabular}
    $$
\begin{lemma}\label{lemma:primaryNumber}
$\prim(G_{p,q})=\frac{1}{2}(q-1)(r+1)$. Thus $\prim(G_{p,q})\geq 2$.
\end{lemma}
\begin{proof}
We establish first the real conjugacy classes of $G_{p,q}$ whose elements are not of prime power order. Let $g\in G$ be such an element. Obviously, we can consider only those $g$ which are the distinguished conjugacy class representatives. It follows from Table \ref{table:classesGpq} that $g\in (B_m)$ or $g\in (A_l)$ for some $m=1,...,q-1$ and $l$ not divisible by $p$ and $q$. In the first case, $g=(b,c^m)$ and $g^{-1}=(b,c^{-m})$, so $(g)\neq (g^{-1})$. This yields $(q-1)/2$ real conjugacy classes of the form $(B_m)^{\pm}=(B_m)\cup(B_{q-m})$ for any $m=1,...,(q-1)/2$. In case $g\in (A_l)$, we have $g=(a^l,1)$ and $g^{-1}=(a^{-l},1)$ and $g$ is conjugate to $g^{-1}$,
$$
(b,1)(a^l,1)(b,1)^{-1}=(a^{-l},1).
$$
Thus, each of the classes $(A_l)$ constitute the real conjugacy class. Therefore $$\prim(G_{p,q})=\frac{1}{2}(q-1)+\frac{1}{2}(q-1)r=\frac{1}{2}(q-1)(r+1).$$
\end{proof}
\subsection{Normal subgroups and quotients of $G_{p,q}$}
\begin{lemma}\label{lemma:ordersNormalSubgroups}
If $N\trianglelefteq G_{p,q}$, then $|N|\in\{1,p,q,pq,2pq,pq^2,2pq^2\}$.
\end{lemma}
\begin{proof}
$|G|$ has the following set of divisors
$$
\{1,2,p,q,2p,2q,pq,q^2,2pq,2q^2,pq^2,2pq^2\}.
$$
We show that $|N|\notin \{2,2p,2q,q^2,2q^2\}$. Assume $2||N|$. Then there is some element of order $2$ in $N$. Since $N$ is a normal subgroup of $G_{p,q}$, it follows from Table \ref{table:classesGpq} that $(B)\subseteq N$ and thus $|N|\geq pq$. Observe that $pq>2,2p,2q,2q^2$. Hence $|N|\notin\{2,2p,2q,2q^2\}$. 

Now, suppose $|N|=q^2$. We conclude from Table \ref{table:classesGpq} that $N\trianglelefteq N_{pq^2}$. However, this possibility was already excluded in the proof of Lemma \ref{lemma:largeNpq2}.
\end{proof}
Consider the following subgroups of $G_{p,q}$.
$$
    \begin{tabular}{cc}
    $N_{2pq}=\{(b^{\varepsilon}a^l,1)|\varepsilon=0,1, l=0,...,pq-1\}$
    &$N_p=\{(a^{qs},1)|s=0,...,p-1\}$\\
    $N_q=\{(a^{ps},1)|s=0,...,q-1\}$&$N_{pq}^1=\{(a^l,1)|l=0,...,pq-1\}$
    \end{tabular}
    $$
and
$$
N_{pq}^2=\{(a^{qs},c^m)|s=0,...,p-1, m=0,...,q-1\}.
$$
\begin{lemma}\label{lemma:uniqueNormal}
$N_{pq^2}, N_{2pq}, N_{pq}^1, N_{pq}^2, N_p$ and $N_q$ are the only proper normal subgroups of $G_{p,q}$. Moreover, $N_{2pq}\cong D_{2pq}$, $N_{pq}^1\cong C_{pq}$, $N_{pq}^2\cong F_{p,q}$, $N_p\cong C_p$ and $N_q\cong C_q$.
\end{lemma}
\begin{proof}
It follows from Table \ref{table:classesGpq} that all the subgroups mentioned in the Lemma consist of the whole conjugacy classes and thus are normal. Clearly, $N_{2pq}\cong D_{2pq}$, $N^1_{pq}\cong C_{pq}$, $N^2_{pq}\cong F_{p,q}$ (since $N_{pq}^2$ is not abelian and the unique nonabelian group of order $pq$ is $F_{p,q}$), $N_p\cong C_p$ and $N_q\cong C_q$. We show that there are no other proper normal subgroups in $G_{p,q}$. 

Assume for the converse that there exists a proper normal subgroup $N$ of $G_{p,q}$ such that $N\notin\{N_p,N_q,N_{pq}^1,N_{pq}^2,N_{2pq},N_{pq^2}\}$. From Lemma \ref{lemma:ordersNormalSubgroups}, we have $$|N|\in\{p,q,pq,2pq,pq^2\}.$$ 

If $|N|=pq^2$, then the only possibility is $N=N_{pq^2}$ which is a contradiction. 

Suppose $|N|=2pq$. Then there exists an element of order $2$ contained in $N$. Thus, $(B)\subseteq N$. Since $N\neq N_{2pq}$, it follows that $g=(x,c^m)\in N$ for some $x\in D_{2pq}$ and $m\neq 0$. Since $(1,1)\in N$, $(B)\subseteq N$ and $|\{(1,1)\}\cup(B)|=pq+1$, it follows that $|(g)|<pq$. We conclude then from Table \ref{table:classesGpq} that $g\in(C_m)$ or $g\in(D_{r,m})$ for some r not divisible by $q$. Thus $C_m\in N$ or $D_{r,m}\in N$. Suppose $D_{r,m}\in N$. Then, for $n\geq 0,$ 
$$
D_{r,m}^n=(a^{r(1+i^m+...+i^{(n-1)m})},c^{nm})
$$
and $D_{r,m}^n\in D_{r_n,(nm \Mod{q})}$ for any $n=1,...,q-1$, where $r_n\neq 0$. Hence $S=(D_{r,m})\cup(D_{r_2,(2m\Mod{q})})\cup...\cup(D_{r_{q-1},((q-1)m\Mod{q})})\subseteq N$. However, $|S|=2p(q-1)>pq$. A contradiction. This means that $C_m\in N$. Therefore $\langle C_m\rangle\leq N$ and thus, for any $m=1,...,q-1$, $(C_m)\subseteq N$. On the other hand, $|(C_1)\cup...\cup(C_{q-1})|=p(q-1)<pq-1$ which means that $D_{r,m}\in N$ for some $m\neq 0$ and $q\nmid r$ which we have already excluded.

Assume $|N|=pq$. Then $N$ has no element of order $2$ and, since $N\neq N_{pq}^1$, $C_m\in N$ or $D_{r,m}\in N$ for some $m\neq 0$ and $q\nmid r$. The latter case implies $|N|>pq$. If $C_m\in N$, then, since $|\{(1,1)\}\cup(C_1)\cup...\cup(C_{q-1})|<pq$, it follows that one of the elements $A_l$ or $F_s$ is contained in $N$ for some $p,q\nmid l$ and $s=1,...,q-1$. If $A_l\in N$, again, we obtain a contradiction for this leads to $|N|>pq$ (for $\langle A_l\rangle=N_{pq}^1$). If $F_s\in N$, then $(F_1)\cup...\cup(F_{q-1})\subseteq N$. On the other hand, there exist an element of order $p$ in $N$ and we conclude from Table \ref{table:classesGpq} that $(E_1)\cup...\cup(E_{p-1})\subseteq N$. Thus,
\begin{eqnarray*}
|N|&\geq&|\{(1,1)\}\cup\bigcup_{r=1}^{q-1}(C_r)\cup\bigcup_{s=1}^{q-1}(F_s)\cup\bigcup_{t=1}^{p-1}(E_t)|\\
&=&1+p(q-1)+2\cdot\frac{1}{2}(q-1)+2q\cdot\frac{1}{2}r=pq+q-1>pq.
\end{eqnarray*}
Thus, $F_s\notin N$ for any $s=1,...,q-1$ and
$$
N=\{(1,1)\}\cup(C_1)\cup...\cup(C_{q-1})\cup(E_1)\cup...\cup(E_{p-1})=N^2_{p,q}
$$
which contradicts our assumption.

Let $|N|=q$ and $g\in N$ be an element of order $q$. If $g\in (C_m)$ or $g\in D_{r,m}$ for some $m\neq 0$ and $q\nmid r$, we conclude from Table \ref{table:classesGpq} that this implies $|N|>q$. Thus, $g\in (F_s)$ which means that it is impossible that $N\neq N_q$. 

If $|N|=p$, Table \ref{table:classesGpq} leads immediately to the contradiction. 
\end{proof}
\begin{corollary}\label{corollary:largeGpq}
$\mathcal{O}^p(G_{p,q})=G_{p,q},\:\mathcal{O}^q(G_{p,q})=N_{2pq}$ and $\mathcal{O}^2(G_{p,q})=N_{pq^2}$. Thus $\mathcal{L}(G_{p,q})=\{N_{2pq},N_{pq^2},G_{p,q}\}$.
\end{corollary}
Since $G_{p,q}$ is the semidirect product of $D_{2pq}$ and $C_q$, it can be presented as follows.
$G_{p,q}=\langle a,b,c|a^{pq},b^2,bab^{-1}=a^{-1},c^q,cac^{-1}=a^i,cbc^{-1}=b\rangle$
and we can identify $a$ with $(a,1)$, $b$ with $(b,1)$ and $c$ with $(1,c)$.
\begin{lemma}\label{lemma:quotientsGpq}
$G_{p,q}/N^1_{pq}\cong C_{2q}$, $G_{p,q}/N_{pq}^2\cong D_{2q}$, $G_{p,q}/N_p\cong C_q\times D_{2q}$ and $G_{p,q}/N_q$ is a group not of prime power order which is not nilpotent.
\end{lemma}
\begin{proof}
Define $\varphi_{pq}^1\colon G_{p,q}\rightarrow C_{2q}=\langle d|d^{2q}=1\rangle$ by $\varphi_{pq}^1(a)=1$, $\varphi_{pq}^1(b)=d^q$, $\varphi_{pq}^1(c)=d^q$. Obviously
$$
\varphi_{pq}^1(a^{pq})=\varphi_{pq}^1(b^2)=\varphi(c^q)=\varphi_{pq}^1(baba)=1
$$
and
$$
\varphi_{pq}^1(cac^{-1})=\varphi_{pq}^1(a^i)=\varphi_{pq}^1(cbc^{-1}b^{-1})=1
$$
and $\varphi_{pq}^1$ is a well-defined group homomorphism. It is easy to observe that $\Ker\varphi_{pq}^1=N^1_{pq}$ and that $\varphi_{pq}^1$ is surjective. Thus $G_{p,q}/N^1_{p,q}\cong C_{2q}$.

Let $\varphi_{pq}^2\colon G_{p,q}\rightarrow D_{2q}=\langle d,e|d^q=e^2=1,ede=d^{-1}\rangle$ be given by $\varphi_{pq}^2(a)=d$, $\varphi_{pq}^2(b)=e$ and $\varphi_{pq}^2(c)=1$. We have  
$$
\varphi_{pq}^2(a^{pq})=\varphi_{pq}^2(b^2)=\varphi(c^q)=\varphi_{pq}^2(baba)=\varphi_{pq}^2(cbc^{-1}b^{-1})=1
$$
and
$$
\varphi_{pq}^2(cac^{-1})=d=d^i=\varphi_{pq}^2(a^i).
$$
Thus $\varphi_{pq}^2$ is a well-defined epimorphism. Moreover, $$b^{\varepsilon}a^lc^m\in\Ker\varphi_{pq}^2\Leftrightarrow e^{\varepsilon}d^l=1\Leftrightarrow \varepsilon=0,q|l,m=0,...,q-1$$
and $\Ker\varphi_{pq}^2=\{(a^{qs},c^m)|s=0,...,p-1,m=0,...,q-1\}=N^2_{pq}$.

Put $\varphi_{p}\colon G_{p,q}\rightarrow C_q\times D_{2q}=\langle d|d^q=1\rangle\times\langle e,f|e^q=f^2=1,fef=e^{-1}\rangle$, $\varphi_{p}(a)=(1,e)$, $\varphi_{p}(b)=(1,f)$ and $\varphi_{p}(c)=(d,1)$. Obviously,
$$
\varphi_{p}(a^{pq})=\varphi_{p}(b^2)=\varphi_{p}(c^q)=\varphi_{p}(baba)=\varphi_{p}(cbc^{-1}b^{-1})=(1,1).
$$
Moreover, $\varphi_{p}(cac^{-1})=(1,e)$ and $\varphi_{p}(a^i)=(1,e^i)$. Since $i\equiv 1\Mod{q}$, it follows that $(1,e^i)=(1,e)$. Hence $\varphi_{p}$ is a well-defined homomorphism. Obviously, $\varphi_{p}$ is surjective and $b^{\varepsilon}a^lc^m\in\Ker(\varphi_{p})\Leftrightarrow m=0,\varepsilon=0,q|l$. Therefore $\Ker(\varphi_{p})=\{a^{qs}|s=0,...,p-1\}=N_p$.

Since the order of $G_{p,q}/N_q$ equals $2pq$, $G_{p,q}/N_q$ is not of prime power order. Suppose for the converse that $G_{p,q}/N_q$ is nilpotent. This means that $G_{p,q}/N_q$ is the direct product of its Sylows and, since $|G_{p,q}/N_q|$ is the product of three distinct primes, we conclude that $G_{p,q}/N_q$ has to be cyclic. Let $dN_q$ be the generator of $G_{p,q}/N_q$. If $d=(ba^l,c^m)$ for some $l=0,...,pq-1$ and $m=0,...,q-1$, then $|d|\leq 2q$ and it follows that $(dN_q)^{2q}=1$ in $G_{p,q}/N_q$. This contradicts that $G_{p,q}/N_q$ is of order $pq$. Thus, $d=(a^l,c^m)$. If $m\neq 0$, then $d^q=(1,1)$ and we obtain a contradiction. Hence, $d=(a^l,1)$. In this case, however, $d^p=(a^{pl},1)\in N_q$ which leads, again, to a contradiction.
\end{proof}
\subsection{$G_{p,q}$ is a special Oliver group}

We will need the following results of Sumi.
\begin{theorem}\label{theorem:Sumi2012}\cite{Sumi2012}[Theorem 1.2]
Let $G$ be a group with no large subgroup of prime power order. Moreover, suppose that $[G\colon\mathcal{O}^2(G)]=2$ and $\mathcal{O}^{p_0}(G)\neq G$ for a unique odd prime $p_0$ and that $G$ does not have an element of order divisible by $4$ and there is an element $g\in G$ of order $2$ not belonging to $\mathcal{O}^2(G)$ such that $2|\mathcal{O}^2(C_G(g))|\geq |C_G(g)|$. Then $G$ is not a gap group. 
\end{theorem}
\begin{lemma}\label{lemma:Sumi2004}\cite{Sumi2004}[p.35, first paragraph]
If $G$ is a group which has a large subgroup of prime power order, then $G$ is not a gap group.
\end{lemma}
\begin{lemma}\label{lemma:sumi2001}\cite{Sumi2001}[pp. 982,984]
For any $n\geq 3$, the dihedral group $D_{2n}$ is not a gap group.
\end{lemma}
Now, we can prove the following.
\begin{lemma}\label{lemma:notGapSubgroups}
$N_{pq}^1$, $N_{2pq}$, $N_{pq^2}$ and $G_{p,q}$ are not gap groups.
\end{lemma}
\begin{proof}
Let us prove that $G_{p,q}$ is not a gap group by means of Theorem \ref{theorem:Sumi2012}. By Corollary \ref{corollary:largeGpq} and the fact that $G_{p,q}$ does not have an element of order divisible by $4$, it suffices to show that there exists an element $g\in G$ of order $2$ not belonging to $\mathcal{O}^2(G)=N_{pq^2}$ such that $2|\mathcal{O}^2(C_G(g))|\geq |C_G(g)|$. We show that this holds for $g=(b,1)$. We have
\begin{eqnarray*}
(b^{\varepsilon}a^l,c^m)(b,1)=(b,1)(b^{\varepsilon}a^l,c^m)\Leftrightarrow (b^{\varepsilon}a^lb,c^m)=(b^{1+\varepsilon}a^l,c^m)
\end{eqnarray*}
which holds if and only if $l=0$. Thus $C_{G_{p,q}}(g)=\{(b^{\varepsilon},c^m)|\varepsilon=0,1,m=0,...,q-1\}\cong C_{2q}$. Obviously $\mathcal{O}^2(C_{2q})\cong C_q$ and therefore the inequality $2|\mathcal{O}^2(C_G(g))|\geq |C_G(g)|$ holds. Hence $G_{p,q}$ is not a gap group.

Note that both $N_{pq}^1$ and $N_{pq^2}$ contain $N_p$ as a normal subgroup (since $N_p\trianglelefteq G_{p,q}$). Thus $\mathcal{O}^q(N_{pq}^1)=\mathcal{O}^q(N_{pq^2})=N_p$ and $N_p$ is a large subgroup for both $N_{pq}^1$ and $N_{pq^2}$. Hence, we get from Lemma \ref{lemma:Sumi2004} that $N_{pq}^1$ and $N_{pq^2}$ are not gap groups.

The statement for $N_{2pq}$ is the direct corollary from Lemma \ref{lemma:sumi2001}.
\end{proof}
\begin{lemma}\label{lemma:notGNil}
$G_{p,q}$ has no cyclic quotient of odd composite order and $G_{p,q}$ does not satisfy the Sumi $G_{p,q}^{\rm nil}$-condition.
\end{lemma}
\begin{proof}
It follows by Lemma \ref{lemma:uniqueNormal} and Lemma \ref{lemma:quotientsGpq} that $G_{p,q}$ has no cyclic quotient of odd composite order. It follows by Lemma \ref{lemma:quotientsGpq} that $G_{p,q}^{\rm nil}=N^1_{pq}$. Assume $xN^1_{pq}=yN_{pq}^1$ for some elements $x,y\in G_{p,q}$ of even order. This means that $x=(ba^l,c^m)$ and $y=(ba^{l'},c^{m'})$ for some $l,l'=0,...,pq-1$ and $m,m'=0,...,q-1$ and
$$
xy^{-1}=(a^{l'i^{m-m'}-l},c^{m-m'})\in N_{pq}^1.
$$
Thus $m'=m$ and $(x)=(y)$ by Table \ref{table:classesGpq}.

Suppose there exist $x',y'\in G_{p,q}$ of composite order such that one of them, say $x'$, is of odd order. Then $x'\in N_{pq}^1$ by Table \ref{table:classesGpq}. Thus, the only subgroups of $G_{p,q}$ which can contain both $x'$ and $y'$ must have $N_{pq}^1$ as a subgroup. These subgroups are precisely $N_{pq}^1$, $N_{2pq}$, $N_{pq^2}$ and $G_{p,q}$. We showed in Lemma \ref{lemma:notGapSubgroups} that they are not gap groups. This shows that $G_{p,q}$ does not satisfy the Sumi $G_{p,q}^{\rm nil}$-condition.
\end{proof}
\begin{lemma}\label{lemma:normalNorCyclicQuotients}
$N_{pq^2}$ has no normal subgroup $P$ of prime power order such that the quotient $N_{pq^2}/P$ is cyclic. The same statement holds for $N_{2pq}$.
\end{lemma}
\begin{proof}
Suppose that $P\trianglelefteq N_{pq^2}$ and $N_{pq^2}/P$ is cyclic. Then $|P|\in\{1,p,q,q^2\}$. If $|P|=1$, then $N_{pq^2}/P\cong N_{pq^2}$ and we obtain a contradiction. The case $|P|=q^2$ is not possible by the proof of Lemma \ref{lemma:largeNpq2}. Let $|P|=q$ and $N_{pq^2}/P=\langle(a^l,c^m)P\rangle$. We know from the proof of Lemma \ref{lemma:quotientsGpq} that $(a^l,c^m)^q=(1,1)$. Thus $|N_{pq^2}/P|\leq q$ which is a contradiction since $|N_{pq^2}/P|=pq$. Therefore $|P|=p$. In this case, however, it follows from Lemma \ref{lemma:isom} and Table \ref{table:charactersNpq2} that $P=N_p$. Suppose $N_{pq^2}/P=\langle(a^l,c^m)P\rangle$. As before, $|N_{pq^2}/P|\leq q$ which is not possible.

Assume that there exists $P\trianglelefteq N_{2pq}$ of prime power order such that $N_{2pq}/P$ is cyclic. Then $|P|\in\{1,2,p,q\}$. Obviously, $P$ cannot be the trivial subgroup and, since there is no normal subgroup of order $2$ in $N_{2pq}$, it follows that $|P|\in\{p,q\}$. This means that $P$ is a subgroup of $N_{pq}^1$. If $|P|=p$, then $P=\{(a^{qs},1)|s=0,...,p-1\}$. Since $|(ba^l,1)|=2$ for any $l=0,...,pq-1$, it follows that $N_{2pq}/P=\langle (a^l,1)P\rangle$. Suppose $(a^l,1)^nP=(ba^{l'},1)P$ for some $n\geq 0$ and $l'=0,...,pq-1$. This means that $(ba^{l'-l},1)=(ba^{l'},1)(a^l,1)^{-1}\in P$. A contradiction which implies that $N_{2pq}/P$ is not cylic. The case $|P|=q$ is analogous.
\end{proof}
\begin{lemma}\label{lemma:specialOliver}
$G_{p,q}$ is a special Oliver group.
\end{lemma}
\begin{proof}
Obviously, $G_{p,q}$ is not of odd order. Since $D_{2pq}$ and $C_q$ are solvable groups, it follows that $G_{p,q}$, as the semidirect product of $D_{2pq}$ and $C_{q}$, is solvable as well. Moreover, by Lemma \ref{lemma:notGNil}, we know that $G_{p,q}$ has no cyclic quotient of odd composite order and does not satisfy the Sumi $G_{p,q}^{\rm nil}$-condition. Thus, we only have to show that $G_{p,q}$ is an Oliver group. Suppose for the converse that this is not true. Then, there exist subgroups $P\trianglelefteq H\trianglelefteq G_{p,q}$ such that $G_{p,q}/H$ and $P$ are of prime power orders and $H/P$ is cyclic. Then, by Lemma \ref{lemma:ordersNormalSubgroups}, $|H|\in\{2pq,pq^2,2pq^2\}$ and thus, by Lemma \ref{lemma:uniqueNormal}, $H\in\{N_{2pq},N_{pq^2},G_{p,q}\}$. However, by Lemmas \ref{lemma:uniqueNormal}, \ref{lemma:quotientsGpq} and \ref{lemma:normalNorCyclicQuotients}, it follows that neither of the groups $N_{2pq}$, $N_{pq^2}$ and $G_{p,q}$ has a normal subgroup of prime power order such that the quotient by it is cyclic. This concludes the proof.
\end{proof}
\section{When $\Ind^G_H\colon\PO(H)\rightarrow\PO(G)$ is a monomorphism?}
Assume $H$ is a subgroup of a group $G$ and consider the induction homomorphism $$\Ind_H^G\colon\RO(H)\rightarrow\RO(G),\; U-V\mapsto \Ind^G_H(U)-\Ind^G_H(V).$$ 
\begin{theorem}\label{theorem:inducedCharacterFormula}\cite[21.23. Theorem]{Liebeck2001}
Let $\chi$ be a character of $H$ and $g\in G$. Then, we have two possibilities.
\begin{enumerate}[(1)]
    \item if $H\cap (g)=\emptyset$, then $\Ind_H^G(\chi)(g)=0$.
    \item if $H\cap (g)\neq\emptyset$, then
    $$
    \Ind_{H}^G(\chi)(g)=|C_G(g)|\Big(\frac{\chi(h_1)}{|C_H(h_1)|}+\ldots+\frac{\chi(h_m)}{|C_H(h_m)|}\Big),
    $$
\end{enumerate}
where $C_K(x)$ denotes the centralizer of the element $x$ of the group $K$ and $h_1,\ldots,h_m$ are the representatives of all the distinct conjugacy classes in $H$ of the elements of the set $H\cap (g)$.
\end{theorem}
The above theorem shows that we have a well-defined restriction
$$
    \Ind^G_H\colon\PO(H)\rightarrow\PO(G).
    $$
Let $s$ denote the number of real conjugacy classes of $G$ which have nonzero intersection with $H$ and whose elements are not of prime power order. Put $t=\prim(H)$ and let $m$ be the number of real conjugacy classes of $G$. Obviously, $s\leq t$. 

Consider the image $\Ima(\Ind_H^G\colon\PO(H)\rightarrow\PO(G))$. Clearly, it is a torsion-free subgroup of $\PO(G)$ and it has a well-defined rank as the minimum $\geq 0$ such that $\Ima(\Ind_H^G\colon\PO(H)\rightarrow\PO(G))\cong\mathbb{Z}^r$. Thus, $\Ind_H^G\colon\PO(H)\rightarrow\PO(G)$ is a monomorphism if and only if $r=t$.

If $A$ is a matrix with entries in the field $K$, then we denote by $\rank_K(A)$ the rank of $A$ over $K$. 
\begin{lemma}\label{lemma:trace}
Assume $A\in\GL(n,\mathbb{C})$ is of finite order. Then $\tr(A^{-1})=\overline{\tr(A)}$.
\end{lemma}
\begin{lemma}\label{lemma:boundForRank}
The rank of $\Ima(\Ind_H^G\colon\PO(H)\rightarrow\PO(G))$ is at most $s$, that is $r\leq s$. 
\end{lemma}
\begin{proof}
Pick the bases $\epsilon=\{\varepsilon_1,...,\varepsilon_{t}\}$ and $\epsilon'=\{\varepsilon_1',...,\varepsilon_{t'}'\}$ of $\PO(H)$ and $\PO(G)$ respectively ($t'=\prim(G)$). We use the column convention for elements from $\PO(H)$ and $\PO(G)$ - we represent them as $t\times 1$ and $t'\times 1$ vectors respectively, where the coordinates are given by the bases $\epsilon$ and $\epsilon'$ accordingly. The induction map is a linear map and  denote by $M$ its matrix form in bases $\varepsilon$ and $\varepsilon'$.

Let $(g_1)^{\pm},...,(g_{t'})^{\pm}$ be the ordered list of all real conjugacy classes of $G$ whose elements are not of prime power order and let $\chi$ be the map which evaluates the characters of the elements of $\PO(G)$ on the classes $(g_i)^{\pm}$ for $i=1,...,t'$. Note by Lemma \ref{lemma:trace} that $\chi$ is well-defined and
$$
\chi\colon\PO(G)\rightarrow\mathbb{R}^{t'}
$$
$$
\varepsilon_j'\mapsto\begin{pmatrix}
    \varepsilon_j'(g_1) \\
    \vdots \\
    \varepsilon_j'(g_{t'})
\end{pmatrix}.
$$
Let $X=(\chi_{ij})_{1\leq i,j\leq t'}$ be the matrix of $\chi$, that is $\chi_{ij}=\varepsilon_j'(g_i)$. Clearly, $\rank_{\mathbb{R}}(X)=t'$. Consider the composition
$$
\chi\circ\Ind^G_H\colon\PO(H)\xrightarrow{\Ind_H^G}\PO(G)\xrightarrow{\chi}\mathbb{R}^{t'}.
$$
The matrix of $\chi\circ\Ind^G_H$ is a $t'\times t$ matrix $A=(a_{ij})_{1\leq i\leq t',1\leq j\leq t}$ given by $A=XM$. Thus, $a_{ij}=\Ind^G_H(\varepsilon_j)(g_i)$ for $1\leq i\leq t'$ and $1\leq j\leq t$. It follows from Theorem \ref{theorem:inducedCharacterFormula} that $\rank_{\mathbb{R}}(A)\leq s$. On the other hand, since $\rank_{\mathbb{R}}(X)=t'$, it follows that $\rank_{\mathbb{R}}(M)=\rank_{\mathbb{R}}(A)$. Therefore $\rank_{\mathbb{R}}(M)\leq s$. 

Now, since $M$ is an integer matrix and $\mathbb{R}$ is an extension of $\mathbb{Q}$, we conclude that the real rank of $M$ equals its rational rank, that is $\rank_{\mathbb{R}}(M)=\rank_{\mathbb{Q}}(M)=r'$. We show that $r=r'$ which would mean that $r\leq s$ and would complete the proof. 

Obviously, $r\geq r'$. Let $V=\langle \varepsilon_1',...,\varepsilon_{t'}'\rangle$. Take any $r'+1$ elements $v_1,...,v_{r'+1}$ from $\Ima(\Ind_H^G\colon\PO(H)\rightarrow\PO(G))$. They can be considered as vectors from $V$. Note that they are linearly dependent (over $\mathbb{Q})$, since the dimension of $\Ima(\Ind_H^G\colon\PO(H)\rightarrow\PO(G))$ considered as a subspace of $V$ equals $r'$. Let
\begin{equation}\label{equation:linearCombination}
\alpha_1v_1+...+\alpha_{r'+1}v_{r'+1}=0
\end{equation}
be a nontrivial combination. Suppose $\{\alpha_{i_1},...,\alpha_{i_k}\}$ is the set of all nonzero coefficients and $\alpha_{i_j}=p_j/q_j$, where $p_j,q_j\in\mathbb{Z}\setminus{\{0\}}$ for $1\leq j\leq k$. Multiplying both sides of equality (\ref{equation:linearCombination}) by $q_1...q_k$, we get a nontrivial integer combination of $v_j$'s. Thus $r\leq r'$ and, as a result, $r=r'$.
\end{proof}
\begin{lemma}\label{lemma:indMonoGeneral}
$\Ind_H^G\colon\PO(H)\rightarrow\PO(G)$ is a monomorphism if and only if $(h)_G^{\pm}\cap H=(h)_H^{\pm}$ for any $h\in H$ not of prime power order.
\end{lemma}
\begin{proof}
Suppose that for any $h\in H$ not of prime power order we have $(h)^{\pm}_G\cap H=(h)^{\pm}_H$. Let $x_1$ and $x_2$ be two different elements of $\PO(H)$. We must show that $\Ind_H^G(x_1)\neq\Ind_H^G(x_2)$. There exists $h\in H$ not of prime power order with $x_1(h)\neq x_2(h)$. We have two possibilities. The first one is when $(h)_H=(h^{-1})_H=(h)_G^{\pm}\cap H$. Thus 
$$(h)_G\cap H=(h^{-1})_G\cap H=(h)^{\pm}_G\cap H=(h)^{\pm}_H=(h)_H$$
and it follows by Theorem \ref{theorem:inducedCharacterFormula} that
$$
\begin{tabular}{ccc}
$
\Ind_H^G(x_1)(h)=\frac{|C_G(h)|}{|C_H(h)|}x_1(h)
$
&
and
&
$
\Ind_H^G(x_2)(h)=\frac{|C_G(h)|}{|C_H(h)|}x_2(h).
$
\end{tabular}
$$
 Therefore $\Ind_H^G(x_1)(h)\neq\Ind_H^G(x_2)(h)$ since $x_1(h)\neq x_2(h)$. In the second possibility, we have $(h)_H\neq(h^{-1})_H$. If $(h)_G\cap H=(h)_H$, we have already proved the assertion. Assume $(h)_H\subsetneq(h)_G\cap H$. Note that $$((h)_G\cap H)\cup((h^{-1})_G\cap H)=(h)_G^{\pm}\cap H=(h)_H\cup(h^{-1})_H.$$ Clearly $(h^{-1})_H\subseteq(h^{-1})_G\cap H$, which in connection with $(h)_H\subsetneq(h)_G\cap H$ gives from the equalities above $(h)_G\cap H=(h^{-1})_G\cap H=(h)_H\cup(h^{-1})_H$. Note that $|C_H(h)|=|C_H(h^{-1})|$. Thus by Theorem \ref{theorem:inducedCharacterFormula} and Lemma \ref{lemma:trace}, we get 
 $$
\begin{tabular}{ccc}
$
\Ind_H^G(x_1)(h)=2\frac{|C_G(h)|}{|C_H(h)|}x_1(h)
$
&
and
&
$
\Ind_H^G(x_2)(h)=2\frac{|C_G(h)|}{|C_H(h)|}x_2(h).
$
\end{tabular}
$$
Thus $\Ind_H^G(x_1)(h)\neq\Ind_H^G(x_2)(h)$.

We prove now the converse. Suppose $\Ind_H^G\colon\PO(H)\rightarrow\PO(G)$ is a monomorphism. Assume for the contrary that there exists $h\in H$ not of prime power order with $(h)^{\pm}_G\cap H\neq(h)^{\pm}_H$. Then $(h)^{\pm}_H\subsetneq(h)^{\pm}_G\cap H$ and thus $s<t$. Hence, it follows by Lemma \ref{lemma:boundForRank} that $\rank(\Ima(\Ind_H^G\colon\PO(H)\rightarrow\PO(G)))<t$ and $\Ind_H^G\colon\PO(H)\rightarrow\PO(G)$ is not injective which is a contradiction with our assumption.
\end{proof}
\begin{corollary}\label{corollary:inducedNormal}
Assume $N$ is a normal subgroup of $G$. Then $\Ind_N^G\colon\PO(N)\rightarrow\PO(G)$ is a monomorphism if and only if $(n)_G^{\pm}=(n)_N^{\pm}$ for any $n\in N$ not of prime power order.
\end{corollary}
\begin{corollary}\label{corollary:inductionNpq2}
$\Ind_{N_{pq^2}}^{G_{p,q}}\colon\PO(N_{pq^2})\rightarrow\PO(G_{p,q})$ is a monomorphism.
\end{corollary}
\begin{proof}
By Corollary \ref{corollary:inducedNormal}, it suffices to show that for any $n\in N_{pq^2}$ not of prime power order, we have $(n)^{\pm}_{G_{p,q}}=(n)^{\pm}_{N_{pq^2}}$. We know by Table \ref{table:classesGpq} that $n$ has to be of order $pq$ and $n=(a^l,1)$ for some $l$ not divisible by $p$ and $q$. From the proof of Lemma \ref{lemma:primaryNumber}, we know that $n^{-1}=(a^{-l},1)$ and $n$ are conjugate in $G_{p,q}$. Thus $(n)^{\pm}_{G_{p,q}}=(n)_{G_{p,q}}$. On the other hand $n$ and $n^{-1}$ are not conjugate in $N_{pq^2}$. Otherwise, there would exists $(a^{l'},c^{m'})$ such that
$$
(a^{l'},c^{m'})(a^l,1)(a^{l'},c^{m'})^{-1}=(a^{-l},1).
$$
Thus $(a^{li^{m'}},1)=(a^{-l},1)$ which cannot be true since $li^{m'}\equiv l\not\equiv -l\Mod{q}$ for $q\nmid l$. Therefore $(n^{-1})_{N_{pq^2}}\neq(n)_{N_{pq^2}}$ and it follows by Table \ref{table:charactersNpq2} that $|(n^{\pm})_{N_{pq^2}}|=2q$. On the other hand, $|(n)^{\pm}_{G_{p,q}}|=|(n)_{G_{p,q}}|=2q$, and the assertion follows.
\end{proof}
\section{Proof of Theorem \ref{theorem:main}}
We use the following result of Morimoto.
\begin{theorem}\cite[Theorem 1.9]{Morimoto2012}\label{theorem:mor}
 Let $H$ be a subgroup of an Oliver group $G$. If $U-V\in\PO_{\rm w}^{\mathcal{L}}(H)$, then there exists an $\mathbb{R}G$-module $W$ such that $\Ind^G_H(U)\oplus W$ and $\Ind^G_H(V)\oplus W$ are Smith equivalent $\mathbb{R}G$-modules.
\end{theorem}
We can prove now the main theorem of this paper.
\begin{proof}[Proof of Theorem \ref{theorem:main}]
 Lemmas \ref{lemma:primaryNumber} and \ref{lemma:specialOliver} tell us that $G_{p,q}$ is a special Oliver group with primary number at least $2$. By Lemma \ref{lemma:noznzeroPO}, there exist non-isomorphic $\mathbb{R}G$-modules $U$ and $V$ with $U-V\in\PO_{\rm w}^{\mathcal{L}}(N_{pq^2})$. Thus, by Corollary \ref{corollary:inductionNpq2}, $\Ind^{G_{p,q}}_{N_{pq^2}}(U)$ and $\Ind^{G_{p,q}}_{N_{pq^2}}(V)$ are not isomorphic $\mathbb{R}G_{p,q}$-modules. Moreover, by Theorem \ref{theorem:mor}, it follows that there exists an $\mathbb{R}G_{p,q}$-module $W$ such that $\Ind^{G_{p,q}}_{N_{pq^2}}(U)\oplus W$ and $\Ind^{G_{p,q}}_{N_{pq^2}}(V)\oplus W$ are Smith equivalent. Obviously, these modules are not isomorphic.
\end{proof}
 \section*{Acknowledgements}
 The author would like to thank the referee for his valuable remarks. Also, I would like to thank Prof. Krzysztof Pawałowski for his comments and interest in the results presented here. I am also grateful to Dr. Marek Kaluba for remarks which substantially improved the presentation of this paper. I would like to thank all the participants of the algebraic topology seminar held at Adam Mickiewicz University for helpful comments.


\newpage
\bibliographystyle{acm}

\textcolor{white}{fsf}\\
\emph{Faculty of Mathematics and Computer Science}\\
\emph{Adam Mickiewicz University in Poznań}\\
\emph{ul Uniwersytetu Poznańskiego 4}\\
\emph{61-614 Poznań, Poland}\\
\emph{Email address:} piotr.mizerka@amu.edu.pl







\end{document}